\documentclass{amsart}
\usepackage{amsmath}

\newtheorem{theorem}{Theorem}[section]
\newtheorem{lemma}[theorem]{Lemma}

\makeatletter
  
  \@addtoreset{equation}{section}
 \makeatother
\usepackage{color}

\def\qedbox{\hbox{$\rlap{$\sqcap$}\sqcup$}}
\def\MM{{\mathfrak{M}}}\def\CC{{\mathfrak{C}}}\def\QQ{{\mathfrak{Q}}}

\begin{document}
  \title[Geometric Realizations of curvature models]{Geometric realizations of curvature models by {manifolds}
with constant scalar curvature}
\author{M.  Brozos-V\'azquez, P. Gilkey, H. Kang, S. Nik\v cevi\'c, G. Weingart}
\address{MB: Department of Mathematics, University of A Coruna, Spain\\
E-mail: mbrozos@udc.es}
\address{PG: Mathematics Department, University of Oregon\\
   Eugene OR 97403 USA\\
   E-mail: gilkey@uoregon.edu}
\address{HK: School of Mathematics, University of Birmingham,
Birmingham {B15 2TT, UK}\\
E-mail: {kangh@bham.ac.uk}}
\address{SN: Mathematical Institute, Sanu,
Knez Mihailova 35, p.p. 367\\
11001 Belgrade,
Serbia\\
E-mail: stanan@mi.sanu.ac.rs}
\address{GW: Instituto de Matem\'aticas (Cuernavaca), Universidad Nacional Aut\'onoma de M\'exico\\
E-mail:gw@matcuer.unam.mx}
\begin{abstract} We show any Riemannian curvature model can be geometrically realized by a manifold with constant
scalar curvature. We also show that any pseudo-Hermitian curvature model, para-Hermitian curvature model,
hyper-pseudo-Hermitian curvature model, or hyper-para-Hermitian curvature model can be realized by a manifold
with constant scalar and $\star$-scalar curvature. 
\\
Keywords: constant scalar curvature, constant $\star$-scalar curvature, geometric realization, 
hyper-para-Hermitian, hyper-pseudo-Hermitian, 
para-Hermitian, pseudo-Hermitian, pseudo-Riemannian. \\ {\it Mathematics Subject
Classification  2000:} 53B20
\end{abstract}
\maketitle

\section{Introduction}
Let $V$ be a finite dimensional real vector space of dimension $m$. One says that $A\in\otimes^4(V^*)$ is an {\it algebraic
curvature tensor} on $V$ if $A$ satisfies the symmetries of the Riemann curvature tensor:
\begin{equation}\label{eqn-1.a}
\begin{array}{l}
A(x,y,z,w)=-A(y,x,z,w)=A(z,w,x,y),\\
A(x,y,z,w)+A(y,z,x,w)+A(z,x,y,w)=0\,.\end{array}
\end{equation}
We say that $\MM:=(V,\langle\cdot,\cdot\rangle,A)$ is a {\it curvature model} if $A$ is an algebraic curvature tensor on $V$
and if
$\langle\cdot,\cdot\rangle$ is a non-degenerate symmetric bilinear form of signature $(p,q)$ on $V$. Two curvature models
${\MM_1=(V_1,\langle\cdot,\cdot\rangle_1,A_1})$ and
${\MM_2=(V_2,\langle\cdot,\cdot\rangle_2,A_2})$ are said to be isomorphic, and one writes $\MM_1\approx\MM_2$, if there
is an isomorphism
$\phi:V_1\rightarrow V_2$ so that
$$
\phi^*\langle\cdot,\cdot\rangle_2=\langle\cdot,\cdot\rangle_1\quad\text{and}\quad
\phi^*A_2=A_1\,.
$$

Let $\MM$ be a curvature model. Let $\varepsilon_{ij}$ and $A_{ijkl}$ be the components of $\langle\cdot,\cdot\rangle$ and
$A$ relative to a basis $\{e_i\}$ for $V$:
$$\varepsilon_{ij}:=\langle e_i,e_j\rangle\quad\text{and}\quad A_{ijkl}:=A(e_i,e_j,e_k,e_l)\,.$$
Let $\varepsilon^{ij}$ be the
inverse matrix. Adopt the {\it Einstein convention} and sum over repeated indices. The components of the {\it Ricci tensor}
$\rho=\rho_\MM$ and the {\it scalar curvature} $\tau=\tau_\MM$ are then given by:
$$\rho_{il}:=\varepsilon^{jk}A_{ijkl},\qquad\text{and}\qquad\tau:=\varepsilon^{il}\varepsilon^{jk}A_{ijkl}\,.$$

\subsection{Pseudo-Riemannian geometry}
Let
$\mathcal{M}:=(M,g)$ be a pseudo-Riemannian manifold of signature
$(p,q)$. Let $\nabla=\nabla_{\mathcal{M}}$ be the Levi-Civita connection of $\mathcal{M}$ and let
$R=R_{\mathcal{M}}\in\otimes^4T^*M$ be the curvature tensor of $\nabla$:
$$R(x,y,z,w)=g((\nabla_x\nabla_y-\nabla_y\nabla_x-\nabla_{[x,y]})z,w)\,.$$
Let $\MM(\mathcal{M},P):=(T_PM,g_P,R_P)$ for $P\in M$ be the corresponding curvature model.
Relating algebraic properties of the curvature tensor to the underlying geometric properties of the manifold is a central
theme in much of differential geometry {--} see, for example, the discussion of Osserman geometry in
\cite{B06,CDGV06,GKVV99,N03}.

 The following
result is well known and shows that the relations of Equation (\ref{eqn-1.a}) generate the universal symmetries of the Riemann
curvature tensor:

\begin{theorem}\label{thm-1.1}
Let $\MM$ be a curvature model. There exists a real analytic pseudo-Riemannian manifold
$\mathcal{M}$ and a point $P$ of $M$ so that $\MM\approx\MM(\mathcal{M},P)$.
\end{theorem}

The following result extends Theorem \ref{thm-1.1} to the category of manifolds with
constant scalar curvature:

\begin{theorem}\label{thm-1.2}
Let $\MM$ be a curvature model. There exists a real analytic pseudo-Riemannian manifold
$\mathcal{M}$ and a point $P$ of $M$ so that $\mathcal{M}$ has constant scalar curvature and so that
$\MM\approx\MM(\mathcal{M},P)$.
\end{theorem}

\subsection{Conformal Geometry}
Let $\MM$ be a curvature model. Let $W=W_\MM$ be the {\it Weyl conformal curvature tensor}.
One says that $\MM$ is {\it conformally flat} if $W_\MM=0$.
\begin{theorem}\label{thm-1.3}
Let $\MM$ be a conformally flat curvature model. There exists a real analytic
conformally flat pseudo-Riemannian manifold
$\mathcal{M}$ and a point $P$ of $M$ so that $\mathcal{M}$ has constant scalar curvature and so that
$\MM\approx\MM(\mathcal{M},P)$.
\end{theorem} 

\subsection{Pseudo-Hermitian and para-Hermitian geometry} Let $J$ be a linear map of $V$ and let \
$\MM=(V,\langle\cdot,\cdot\rangle,A)$ be a curvature model.
One says that $J$ is a {\it pseudo-Hermitian structure} if
$$J^2=-\operatorname{id}\quad\text{and}\quad J^*\langle\cdot,\cdot\rangle=\langle\cdot,\cdot\rangle\,.$$
Similarly, one says that $J$ is a {\it para-Hermitian structure} if
$$J^2=\operatorname{id}\quad\text{and}\quad J^*\langle\cdot,\cdot\rangle=-\langle\cdot,\cdot\rangle\,.$$
Note that pseudo-Hermitian structures exist if and only if both $p$ and $q$ are
even; para-Hermitian structures exist if and only if $p=q$. Let
${\CC:=(V,\langle\cdot,\cdot\rangle,J,A)}$ be the associated {\it pseudo-Hermitian curvature model} (resp. {\it
para-Hermitian curvature model}). In either case,
define the {\it
$\star$-scalar curvature} $\tau^\star=\tau^\star_\CC$ by setting
$$\tau^\star:=\left\{
\begin{array}{lll}
\phantom{-}\varepsilon^{il}\varepsilon^{jk}A(e_i,e_j,Je_k,Je_l)&\text{if}&\CC\text{ is pseudo-Hermitian},\\
-\varepsilon^{il}\varepsilon^{jk}A(e_i,e_j,Je_k,Je_l)&\text{if}&\CC\text{ is para-Hermitian}\,.
\end{array}\right.$$

One says that $\mathcal{C}:=(M,g,J)$ is an {\it almost pseudo-Hermitian manifold} (resp. {\it almost para-Hermitian
manifold}) if ${\CC(\mathcal{C},P):=(T_PM,g_P,J_P,R_P)}$ is a pseudo Hermitian (resp. para-Hermitian) curvature model for
every
$P\in M$. We do not assume that the structure $J$ on $M$ is integrable as this imposes additional curvature identities
\cite{gray}; we will return to this question in a subsequent paper. Almost pseudo-Hermitian geometry has been studied
extensively. We refer to
\cite{CFG} for further information concerning almost para-Hermitian geometry as it is important as well. For example,
para-Hermitian geometry enters in the study of Osserman Walker metrics of signature $(2,2)$ \cite{nuevo}, it is important
in the study of homogeneous geometries
\cite{Gada}, and it is relevant to the study of Walker manifolds with degenerate self-dual Weyl curvature operators
\cite{Alex}. We refer to \cite{FFS} for information concerning almost-Hermitian geometry.

\begin{theorem}\label{thm-1.4}
Let $m\ge4$. Let ${\CC=(V,\langle\cdot,\cdot\rangle,J,A)}$ be a pseudo-Hermitian (resp. para-Hermitian) curvature model.
There exists a real analytic almost pseudo Hermitian (resp. almost para-Hermitian) manifold
$\mathcal{C}=(M,g,J)$ and a point $P$ of $M$ so that $\mathcal{C}$ has constant scalar curvature, so that $\mathcal{C}$ has
constant {$\star$}-scalar curvature, and so that
${\CC\approx(T_PM,g_P,J_P,R_P)}$.
\end{theorem}

\subsection{Hyper-pseudo-Hermitian and hyper-para-Hermitian geometry} Fix a curvature model
$\MM=(V,\langle\cdot,\cdot\rangle,A)$. Let
$\mathcal{J}:=\{J_1,J_2,J_3\}$ be a triple of linear maps of $V$. We say that $\mathcal{J}$ is a {\it hyper-pseudo-Hermitian
structure} if $J_1$, $J_2$, $J_3$ are pseudo-Hermitian structures and if we have the quaternion identities:
$$J_1^2=J_2^2=J_3^2=-\operatorname{id}\quad\text{and}\quad J_1J_2=-J_2J_1=J_3\,.$$
Similarly, we say that $\mathcal{J}$ is a {\it hyper-para-Hermitian structure} if $J_1$ is a pseudo-Hermitian structure, if
$J_2$ and $J_3$ are para-Hermitian structures, and if we have the para-quaternion identities:
$$J_1^2=-J_2^2=-J_3^2=-\operatorname{id}\quad\text{and}\quad J_1J_2=-J_2J_1=J_3\,.$$
Let ${\QQ:=(V,\langle\cdot,\cdot\rangle,\mathcal{J},A)}$ be the associated
{\it hyper-pseudo-Hermitian curvature model} (resp. {\it hyper-para-Hermitian curvature model)}. We
refer to \cite{CS04,IZ2005,Kam99} for further details concerning such structures. We define:
$$\tau^\star_\QQ:=\tau^\star_{J_1}+\tau^\star_{J_2}+\tau^\star_{J_3}\,.$$

The structure group of a hyper-pseudo-Hermitian structure $\mathcal{J}$ is $SO(3)$ and of a hyper-para-Hermitian structure is
$SO(2,1)$ since we must allow for reparametrizations; $\tau^\star_\QQ$ is invariant under this structure group and does not
depend on the particular parametrization chosen. {We say that $(M,g,\mathcal{J})$ is a {\it hyper-pseudo-Hermitian} manifold 
or a {\it hyper-para-Hermitian manifold} if $\mathcal{J}_P$ defines the appropriate structure on $(T_PM,g_P)$ for all points
$P$ of
$M$.}

\begin{theorem}\label{thm-1.5}
Let $m\ge8$. Let ${\QQ=(V,\langle,\cdot\rangle,\mathcal{J},A)}$ be an hyper-pseudo-Hermitian (resp.
hyper-para-Hermitian) curvature model. There exists a real analytic almost hyper-pseudo-Hermitian (resp. almost
hyper-para-Hermitian) manifold
$\mathcal{Q}$ and a point $P$ of $M$ so that $\mathcal{Q}$ has constant scalar curvature, so that $\mathcal{Q}$ has constant
{$\star$}-scalar curvature, and so that
${\QQ\approx(T_PM,g_P,\mathcal{J}_P,R_P)}$.
\end{theorem}

{The problems we are considering are related to the {\it Yamabe problem} where one seeks to find a Riemannian metric of
constant scalar curvature in the conformal class of a given compact Riemannian manifold of dimension $m\ge 3$; this has been
solved \cite{A76,R84,T68,Y60}. The complex analogue of the Yamabe problem is to find an almost Hermitian metric of constant
scalar curvature in the conformal class of a given compact almost Hermitian manifold of dimension $m\ge 4$; this problem also
has been solved \cite{RS03}. Our setting is quite different as we wish to fix the curvature tensor at a point and thus we work
purely locally.}

\subsection{Outline of the paper} In Section \ref{sect-2}, we
review the Cauchy-Kovalevskaya Theorem as this is central to our discussion. In Section \ref{sect-3}, we prove
Theorems
\ref{thm-1.1},
\ref{thm-1.2}, and \ref{thm-1.3}. In Section
\ref{sect-4}, we prove Theorems \ref{thm-1.4} and \ref{thm-1.5}.

\section{The Cauchy-Kovalevskaya Theorem}\label{sect-2}
In this Section, we state the version of the Cauchy-Kovalevskaya Theorem that we shall need; we refer to Evans \cite{E}
pages 221-233 for the proof. Introduce coordinates $x=(x_1,...,x_m)$ on $\mathbb{R}^m$ and let
$\partial_i:=\frac{\partial}{\partial x_i}$. Set
$x=(y,x_m)$ where
$y=(x_1,...,x_{m-1})\in\mathbb{R}^{m-1}$. Let
$W$ be an auxiliary real vector space. In Section \ref{sect-3}, we will take $W=\mathbb{R}$ to consider a single scalar
equation and in Section
\ref{sect-4}, we will take $W=\mathbb{R}^2$ to consider a pair of scalar equations to deal with both the scalar curvature and
{$\star$}-scalar curvature. Let
$$u:=(u_0,u_1,...,u_m)\in W\otimes\mathbb{R}^{m+1}\,.$$
We suppose given a real analytic function $\psi(x,u)$ taking values in
$W$ and a collection of real analytic functions $\psi^{ij}(x,u)=\psi^{ji}(x,u)$ taking values in $\operatorname{End}(W)$ which
are defined near
$0$. Given a real analytic function
$U:\mathbb{R}^m\rightarrow W$ which is defined near $x=0$, one sets $u(x):=(u_0(x),...,u_m(x))$ where
$$u_0(x):=U(x),\quad
u_1(x):=\partial_1U(x),\quad...{,}\quad u_m(x):=\partial_mU(x)\,.
$$

\begin{theorem}\label{thm-2.1}
{\bf [Cauchy-Kovalevskaya]} If $\det\psi^{mm}(0)\ne0$, there is $\varepsilon>0$ and a unique real analytic $U$
defined {for}
$|x|<\varepsilon$ which satisfies the following equations:
\begin{eqnarray*}
&&\psi^{ij}(x,u(x))\partial_i\partial_jU(x)
+\psi(x,u(x))=0,\\
&&U(y,0)=0,\quad\text{and}\quad\partial_mU(y,0)=0\,.\
\end{eqnarray*}
\end{theorem}

\section{The Proof of Theorems \ref{thm-1.1}-\ref{thm-1.3}}\label{sect-3}

Although Theorem \ref{thm-1.1} is well known, we give the proof for the sake of completeness. Let $M$ be
a small neighborhood of $0\in V$, let $P=0$, let $(x_1,...,x_m)$ be the system of local coordinates on $V$ induced by a
basis $\{e_i\}$ for $V$, and let
$$
\textstyle g_{ik}:=\varepsilon_{ik}-\frac13A_{ijlk}x^jx^l\,.
$$
Clearly $g_{ik}=g_{ki}$. As $g_{ik}(0)=\varepsilon_{ik}$ is non-singular,  $g$ is a
pseudo-Riemannian metric on some neighborhood of the origin. Let $g_{ij/k}:=\partial_kg_{ij}$ and
$g_{ij/kl}:=\partial_k\partial_lg_{ij}$. The Christoffel symbols of the first kind are:
$$
\Gamma_{ijk}:=g(\nabla_{\partial_i}\partial_j,\partial_k)
       =\textstyle\frac12(g_{jk/i}+g_{ik/j}-g_{ij/k})\,.
$$
As $g=\varepsilon+O(|x|^2)$ and $\Gamma=O(|x|)$, we complete the proof of Theorem \ref{thm-1.1} by computing:
\medbreak\qquad
$R_{ijkl}=\{\partial_i\Gamma_{jkl}-\partial_j\Gamma_{ikl}\}+O(|x|^2)$
\medbreak\qquad\qquad
$=\textstyle\frac12\{g_{jl/ik}+g_{ik/jl}
    -g_{jk/il}-g_{il/jk}
    \}+O(|x|^2)$
\medbreak\qquad\qquad
$=
 \textstyle\frac16\{-A_{jikl}-A_{jkil}-A_{ijlk}-A_{iljk}$
\smallbreak\qquad\qquad\qquad
$   +A_{jilk}+A_{jlik}+A_{ijkl}+A_{ikjl}\}+O(|x|^2)$
\medbreak\qquad\qquad
$=\textstyle\frac16\{4A_{ijkl}-2A_{iljk}-2A_{iklj}\}+O(|x|^2)$
\medbreak\qquad\qquad
$ =A_{ijkl}+O(|x|^2)$.\hfill\qedbox

\medbreak The following fact will be used in the proof of Theorem \ref{thm-1.2}. Again, we include the proof for the
sake of completeness.
\begin{lemma}\label{lem-3.1}
Let $\MM$ be a conformally flat curvature model. There exists a pseudo-Riemannian manifold
$\mathcal{M}$ and a point $P$ of $M$ so that $\mathcal{M}$ is conformally flat, and so that
$\MM\approx\MM(\mathcal{M},P)$.
\end{lemma}

\begin{proof} If $A$ is conformally flat, then $A$ is completely determined by its Ricci tensor. Let
$g:=(1+\phi(x))\langle\cdot,\cdot\rangle$ where $\phi$ is quadratic. The metric $g$ is non-singular for $x$ small, $g$ is
conformally flat, and
$\phi$ can be chosen appropriately so that $\rho(0)=\rho_\MM$:
{$$\phi=\sum_j\frac{\varepsilon_{jj}\tau+(2-2m)\rho_{\MM,jj}}{2(m-1)(m-2)}x_j^2
+\sum_{i<j}\frac{2}{2-m}\rho_{\MM,ij}x_ix_j\,.$$}
The proof now follows. \end{proof}

If $\mathcal{M}$ is a pseudo-Riemannian manifold and if $\phi$ is a smooth function so that $1+2\phi$ never vanishes,
we can consider the conformal variation 
$$\mathcal{M}_\phi=(M,(1+2\phi)g)\,.$$
The metrics constructed to prove Theorem \ref{thm-1.1} and Lemma \ref{lem-3.1} were quadratic polynomials and hence real
analytic.  Theorem \ref{thm-1.2} and Theorem \ref{thm-1.3} will follow from Theorem \ref{thm-1.1} and from Lemma \ref{lem-3.1},
respectively, and from the following result which is perhaps of interest in its own right:

\begin{theorem}\label{thm-3.2}
Let $\mathcal{M}$ be a real analytic pseudo-Riemannian manifold. Fix a point $P$ of $M$. There exists an open neighborhood
$\mathcal{O}$ of
$P$ in
$M$ and a real-analytic function $\phi$ so that $1+2\phi>0$ on $\mathcal{O}$, so that $(\mathcal{O},(1+2\phi)g)$ has constant
scalar curvature, and so that $\MM(\mathcal{O},(1+2\phi)g,P)=\MM(\mathcal{O},g,P)$.
\end{theorem}

\begin{proof} Let $R$ be the curvature tensor of $g$ and
let $\tau$ be the scalar curvature of $g$. Let $x=(x_1,...,x_m)$ be a system of local real analytic coordinates on $M$ centered
at
$P$ and let $y=(x_1,...,x_{m-1})$. Let
$\varepsilon_{ij}:=g(\partial_i,\partial_j)(0)$. By
making a linear change of coordinates, we may suppose that $\{\partial_i\}$ is an orthonormal frame at $P$, or, in other
words, that
$$\varepsilon_{ij}=\left\{\begin{array}{lll}
0&\text{if}&i\ne j,\\
\pm1&\text{if}&i=j\,.\end{array}\right.
$$
Let
$\phi$ be a real analytic function. We set $\phi_i:=\partial_i\phi$ and $\phi_{ij}:=\partial_i\partial_j\phi$. We assume
$$\phi(y,0)=0\quad\text{and}\quad \phi_m(y,0)=0\,.$$
We consider the conformal variation
$
h:=(1+2\phi)g$.
Since $\phi(0)=0$, $h$ is non-singular on some neighborhood of $0$.
Let  $\tilde R$ be the curvature tensor of $h$ and let $\tilde\tau$ be the scalar curvature of $h$. We work modulo terms
$\psi(x,\phi,\phi_1,...,\phi_m)$ where $\psi(0)=0$ to define an equivalence relation $\equiv$. Then
\begin{eqnarray*}
&&\tilde R_{ijkl}\equiv R_{ijkl}+g_{jl}\phi_{ik}-g_{il}\phi_{jk}
-g_{jk}\phi_{il}+g_{ik}\phi_{jl},\\
&&\tilde\tau-\tau_g(0)\equiv h^{il}h^{jk}\{g_{jl}\phi_{ik}-g_{il}\phi_{jk}
-g_{jk}\phi_{il}+g_{ik}\phi_{jl}\}\,.
\end{eqnarray*}
We set $h^{jk}=\varepsilon^{jk}$
and compute
$$
\varepsilon^{il}\varepsilon^{jk}\{\varepsilon_{jl}\phi_{ik}-\varepsilon_{il}\phi_{jk}
-\varepsilon_{jk}\phi_{il}+\varepsilon_{ik}\phi_{jl}\}
\equiv\varepsilon^{ik}\phi_{ik}-m\varepsilon^{jk}\phi_{jk}-m\varepsilon^{il}\phi_{il}+\varepsilon^{jl}\phi_{jl}\,.
$$
The coefficient of $\phi_{mm}$ is thus seen to be $(2-2m)\varepsilon^{mm}\ne0$. Consequently, Theorem \ref{thm-2.1} is
applicable and we may choose
$\phi$ to solve the equations:
$$
\tilde\tau-\tau_g(0)=0,\quad\phi(y,0)=0,\quad\partial_m\phi(y,0)=0\,.
$$

The
$0$ and $1$ jets of $\phi$ vanish at the origin. And the only possibly non-zero $2$-jet of $\phi$ at the origin is
$\phi_{mm}$. The relation $\psi^{ij}\phi_{ij}\equiv0$ implies
$\psi^{mm}\phi_{mm}(0)=0$. Thus all the $2$-jets of $\phi$ vanish at the origin so
$\tilde R(0)=R(0)$ and $h(0)=g(0)$. Theorem \ref{thm-3.2} now
follows.\end{proof}

\section{The proof of Theorems \ref{thm-1.4} and \ref{thm-1.5}}\label{sect-4}

We begin our discussion by normalizing the $2$-jets appropriately:
\begin{lemma}\label{lem-4.1}
\ \begin{enumerate}
\item Let ${\CC=(V,\langle\cdot,\cdot\rangle,J,A)}$ be a pseudo-Hermitian (resp. para-Hermitian) curvature model. There
exists a real analytic almost pseudo-Hermitian (resp. almost para-Hermitian) manifold
$\mathcal{C}=(M,g,J)$ and a point $P$ of $M$ so
${\CC\approx(T_PM,g_P,J_P,R_P)}$.
\item Let $\QQ=(V,\langle\cdot,\cdot\rangle,A,\mathcal{J})$ be an hyper-pseudo-Hermitian (resp.
hyper-para-Hermitian) curvature model. There exists a real analytic almost hyper-pseudo-Hermitian (resp. almost
hyper-para-Hermitian) manifold
$\mathcal{Q}$ and a point $P$ of $M$ so
${\QQ\approx(T_PM,g_P,\mathcal{J}_P,R_P)}$.
\end{enumerate}\end{lemma}

\begin{proof} We consider the squaring map $T:\Psi\rightarrow\Psi^2$ mapping $M_m(\mathbb{R})\rightarrow M_m(\mathbb{R})$.
We localize at the point $\Psi=\operatorname{id}$ and express $(1+\phi)\rightarrow(1+2\phi+\phi^2)$ to see the Jacobean is
multiplication by $2$ and hence invertible. Thus by the inverse function theorem, there is a real analytic map
$S:M_m(\mathbb{R})\rightarrow M_m(\mathbb{R})$ defined near $\operatorname{id}$ so $S(\Psi)^2=\Psi$. Furthermore if
$\psi^2=\Psi$ and if $\psi$ is close to $\operatorname{id}$, then $\psi=S(\Psi)$.

{Suppose given a complex model
$\CC=(V,\langle\cdot,\cdot\rangle,J,A)$}. Set $\varrho=-1$ if $\CC$ is pseudo-Hermitian and
$\varrho=+1$ if $\CC$ is para-Hermitian. 
We use Theorem
\ref{thm-1.1} to choose an analytic pseudo-Riemannian metric $g$ so that $g_P=\langle\cdot,\cdot\rangle$ and $R_P=A$. The
difficulty now is to extend
$J$ to be a suitable structure $J_1$ on $TM$. First extend $J$ and $\langle\cdot,\cdot\rangle$ to a neighborhood of $P$
to be constant with respect to the coordinate frame. Express $g(x,y)=\langle\Psi x,y\rangle$ for $\Psi$ a real
analytic map defined near $P$ taking values in
$M_m(\mathbb{R})$ with $\Psi(P)=\operatorname{id}$. Let $\psi=S(\Psi)$. Since $\Psi^*=\Psi$, $\psi^*=\psi$.
Consequently $g(x,y)=\langle\psi x,\psi y\rangle$ so $g=\psi^*\langle\cdot,\cdot\rangle$.
Set $J_1:=\psi^{-1}J\psi=\psi^*J$. Then
\begin{eqnarray*}
&&J_1^2=(\psi^*J)^2={\psi^{-1}J\psi\psi^{-1}J\psi={\varrho}\operatorname{id}},\\
&&J_1^*g=(\psi^*J)^*\{\psi^*\langle\cdot,\cdot\rangle\}=\psi^*\{J^*\langle\cdot,\cdot\rangle\}=-\psi^*\varrho\langle\cdot,\cdot\rangle
=-{\varrho }g\,.
\end{eqnarray*}
Thus $(M,g,J_1)$ provides the required structure. Assertion (1) follows; we use the same
construction to prove Assertion (2).\end{proof}

Let $\mathcal{C}:=(M,g,J)$ be an almost pseudo-Hermitian [$\varrho=-1$] or an almost
para-Hermitian $[\varrho=+1]$ manifold. Let $2r=m$ and let $\{x_1,...,x_{2r}\}$ be coordinates centered at
$P\in M$ so that
$\{\partial_i\}$ form an orthonormal frame at $P$ and so
$$J(\partial_i)=\left\{\begin{array}{rll}
\partial_{i+r}&\text{if}&i\le r,\\
\varrho\partial_{i-r}&\text{if}&r<i\le m=2r\,.
\end{array}\right.$$
We consider an almost pseudo-Hermitian  (resp. almost para-Hermitian) variation
$$h_{\xi,\eta}:=g+2\xi\{dx_1\circ dx_1-\varrho Jdx_1\circ Jdx_1\}+2\eta\{dx_m\circ dx_m-\varrho Jdx_m\circ Jdx_m\}$$
where $\xi(P)=0$ and $\eta(P)=0$. Theorem \ref{thm-1.4} will follow from Lemma \ref{lem-4.1} and from:

\begin{theorem}\label{thm-4.2} 
Let $(M,g,J)$ be a real analytic almost pseudo-Hermitian (resp. almost para-Hermitian) manifold. Fix $P$ in $ M$. There exists
an open neighborhood
$\mathcal{O}$ of
$P$ in
$M$ and there exist $\xi,\eta\in C^\infty(\mathcal{O})$ so that:
\begin{enumerate}
\item \{$\xi,\eta\}$ vanish to second order at $P$.
\item Both $\tau$ and $\tau^\star$ are constant for $(\mathcal{O},h_{\xi,\eta},J)$.
\end{enumerate}\end{theorem}

Note that by (1), $h=h_{\xi,\eta}$ is non-singular near $P$ and $R_h(P)=R_g(P)$.

\begin{proof} {If $h=g+2\Theta$, we have}
$$R_{ijkl}=\Theta_{ik/jl}+\Theta_{jl/ik}-\Theta_{il/jk}-\Theta_{jk/il}+...\,.$$
Thus the non-zero curvatures of interest are, up to the usual $\mathbb{Z}_2$ symmetries,
$${R_{mrrm}=\varrho\eta_{mm}+...,\quad
  R_{m11m}=-\xi_{mm}+...,\quad
  R_{m,r+1,r+1,m}=\varrho\xi_{mm}+...\,.}$$
This leads to the same formulas in both the pseudo-Hermitian and in the para-Hermitian settings:
$$\begin{array}{ll}
\tau=-4\varepsilon^{11}\varepsilon^{mm}\xi_{mm}-2\eta_{mm}+...,\\
\tau^\star=\phantom{-}0\varepsilon^{11}\varepsilon^{mm}\xi_{mm}-2\eta_{mm}+...\,.
\end{array}$$
These two equations are linearly independent. Consequently the vector valued version of the Cauchy-Kovalevskaya theorem implies
we can solve 
$$\tau^h-\tau^g(0)=0\quad\text{and}\quad\tau^{\star,h}-\tau^{\star,g}(0)=0$$
with
${\xi(y,0)=\xi_m(y,0)}=\eta(y,0)=\eta_m(y,0)=0$. Again, the only possible non-zero $2$-jet is $\eta_{mm}$ and $\xi_{mm}$
and those are seen to be zero by the equation.
\end{proof}

The proof of Theorem \ref{thm-1.5} follows similar lines. Let $J_i^2=\varrho_i\operatorname{id}$. We may decompose
$V=V_1\oplus...\oplus V_\ell$ where $4\ell=m$ and where each $V_i$ is invariant under the structure $\mathcal{J}$. We set
$$\Xi_i:=dx_i\circ dx_i-\varrho_1J_1^*dx_i\circ J_1^*dx_i-\varrho_2J_2^*dx_i\circ J_2^*dx_i-\varrho_3J_3^*dx_i\circ
J_3^*dx_i\,.$$ We then consider
variations of the form $h_{\xi,\eta}:=g+2\xi\Xi_1+2\eta\Xi_m$.
It is then immediate that $h$ is invariant under the action of $\mathcal{J}$. We prove Theorem \ref{thm-1.5} by computing:
\medbreak\qquad
$\tau=-8\varepsilon^{11}\varepsilon^{mm}\xi_{mm}-6\eta_{mm}+...$,
\qquad
$\tau^\star=0\xi_{mm}-6\eta_{mm}+....$\hfill\qedbox

\section*{Acknowledgments} The research of M. Brozos-V\'azquez and of P. Gilkey partially supported by Project MTM2006-01432
(Spain). Research of H. Kang partially supported by the University of Birmingham (UK). Research of S. Nik\v cevi\'c partially
supported by Research of Project 144032 (Serbia). Research of G. Weingart is supported by PAPIIT (UNAM) through research
project IN115408.

\end{document}